\documentclass[ECP]{ejpecp} 

\SHORTTITLE{Long-term investment for prices with negative memory}

\TITLE{Optimal long-term investment in illiquid markets when prices have negative memory\thanks{Supported by
the ``Lend\"ulet'' grant LP 2015-6 of the Hungarian Academy of Sciences.}} 

\AUTHORS{%
  L\'or\'ant Nagy\footnote{Central European University and R\'enyi Institute, Budapest, Hungary.
    \EMAIL{lorantnagy28@gmail.com}}
  \and 
  Mikl\'os R\'asonyi\footnote{R\'enyi Institute, Budapest, Hungary. \EMAIL{rasonyi@renyi.hu} }}


\AMSSUBJ{91G10, 91G80} 

\SUBMITTED{15th May, 2020} 
\ACCEPTED{17th March, 2021} 

\ARXIVID{2005.07080} 

\VOLUME{0}
\YEAR{2020}
\PAPERNUM{0}
\DOI{0}

\ABSTRACT{In a discrete-time financial market model with instantaneous price impact, we find an asymptotically optimal
strategy for an investor maximizing her expected wealth. The asset price is assumed to follow
a process with negative memory. We determine how the optimal growth rate depends 
on the impact parameter and on the covariance decay rate of the price.}

\begin{document}

\section{Introduction}

Fractional Brownian motions (FBMs) with various Hurst parameters $H\in (0,1)$ have 
been enticing researchers of financial mathematics for a long time, since the appearance
of \cite{mandelbrot} where such models for asset prices were suggested for the first time. 
Econometric literature eagerly investigated related models, see \cite{baillie} for an early survey. 

In idealistic models of trading, where market imperfections are disregarded,
FBMs do not provide admissible models since they generate arbitrage
opportunities (for $H\neq 1/2$), see \cite{rogers}, hence they quickly fell out of favour. 
Subsequent research revealed, however, that in the presence of market frictions, arbitrage disappears and FBMs
become eligible candidates for describing prices, see \cite{guasoni2006no} and \cite{gr}. 

Not only prices but also volatilites were also successfully modelled using FBMs, see e.g.\ \cite{cr,gjr}.
Note that, though market volatility is not an asset, trading so-called VIX futures is essentially
equivalent to ``trading'' the volatility, see \cite{tim}.


In markets with instantaneous price impact the first analysis of long-term investment 
for an asset price following an FBM has been carried out in \cite{gnr}: the
optimal growth rate of expected portfolio wealth has been found and an asymptotically optimal strategy has been exhibited.
The robustness of such results was the next natural question:
is the particular structure of FBMs needed for these conclusions? In \cite{gnr}
a larger class of Gaussian processes could also be treated where future increments are
positively correlated to the past and the covariance structure is similar to that of FBMs with $H>1/2$. 
The question of extending the case of FBMs with $H<1/2$ to more general models remained open.

The current paper provides such an extension, based on more involved estimates than in the positively
correlated case. For simplicity, we stay in a discrete-time setting. We derive conclusions
similar to those of \cite{gnr} in the case $H<1/2$, but this time for a larger class of Gaussian processes.

These models allow negative price values. As such, they can directly describe
futures contracts. They can also be considered as stylized models reflecting important characteristics
of a more general class of processes. Our theoretical results then give hints to approach more
complex and realistic models as well. The situation is similar to that of the Bachelier model
(allowing negative prices) and the Black-Scholes model: for many practical purposes they are
equivalent, see e.g.\ \cite{walter}.

Our results are asymptotic in the sense that conclusions are derived for long horizons,
as in several standard settings of stochastic optimal control, see e.g.\ \cite{hl}.
It should be noted that such long horizons indeed arise e.g.\ in high-frequency trading
where investment strategies are executed at every millisecond and the trading interval
is several hours long every day. In the more usual trading regime (with actions taken every day or
every hour), the simulations of \cite{gnr} show that
our strategies perform as expected already on realistic time horizons (say, one year).

\section{Market model}\label{secMod}

Let $(\Omega,\mathcal{F},P)$ be a probability space equipped with a 
filtration $\mathcal{F}_t$, $t\in \mathbb{Z}$. Let $E[X]$ denote the expectation of a real-valued
random variable $X$ (when exists). 

We will consider a financial market 
where the price of a risky asset follows a process $S_t$, $t\in \mathbb{N}$, adapted to $\mathcal{F}_{t}$,
$t\in\mathbb{N}$. The riskless asset is assumed to have price constant $1$.

Realistic modelling needs to take into account market frictions.
We choose to be working in a market model with a temporary, nonlinear price impact where portfolios are penalized
by the integral of a certain power of the trading \emph{speed}. Such models were considered in
\cite{almgren-chriss,bertsimas-lo,garleanu-pedersen,kyle,rs} in continuous time and in \cite{dolinsky-soner}
in discrete time. These models subsume the popular choices of linear ($\alpha=2$ in the discussion below)
and square-root ($\alpha=3/2$) price impact.
 
For some $T\in\mathbb{N}$ 
the class of \emph{feasible strategies} up to terminal time $T$
is defined as 
\begin{equation}
\mathcal{S}(T):=\left\{\phi = (\phi_t)_{t = 0}^T:\phi\text{ is an $\mathbb{R}$-valued, adapted process}\right\}.
\end{equation}
As we will see, $\phi_{t}$ represents the \emph{change} in the investor's position in the given asset (the ``speed''
of trading, to emphasize the analogy with continuous-time models).
Let $z=(z^0,z^1)\in\mathbb{R}^{2}$ be a deterministic initial endowment where $z^{0}$ is in cash and $z^{1}$
is in the risky asset. 

For a feasible strategy $\phi\in\mathcal{S}(T)$, the number of shares in the risky asset, with $\Phi_0 = z^1$, at any time $t\geq 1$, is equal to
\begin{eqnarray}\label{eq:w}
\Phi_{t} := z^1+\sum_{u=0}^{t-1} \phi_u\,.
\end{eqnarray}

For simplicity, we assume $z^1=0$ from now on, i.e. the initial number of shares is zero. We will shortly derive a similar formula for the cash position of the investor.
In classical, frictionless models of trading, cash at time $T+1$ equals 
\begin{equation}\label{cukcsi}
\sum_{u=1}^{T+1} \Phi_{u} \left(S_u - S_{u-1} \right)
\end{equation}
when starting from a $0$ initial position. Algebraic manipulation of \eqref{cukcsi} yields 
$$
\sum_{u=1}^{T+1} \Phi_{u} \left(S_u - S_{u-1} \right)={}
-\sum_{u=0}^{T} \phi_u S_u + S_{T+1} \sum_{u=0}^T \phi_u.{}
$${}

We assume that price impact is a superlinear power function of the ``trading speed'' $\phi$ so we
augment the above with a term that implements the effect of friction:
$$-\sum_{u=0}^{T} \phi_u S_u + S_{T+1} \sum_{u=0}^T \phi_u  - \sum_{u=0}^T \lambda |\phi_u|^{\alpha}$$ 
where we assume $\alpha>1$ and $\lambda>0$. 
We wish to utilize only those portfolios where the risky asset is liquidated by the end of 
the trading period so we define 
$$
\mathcal{G}(T):=\left\{\phi\in\mathcal{S}(T):\Phi_{T+1}=\sum_{u=0}^T \phi_u = 0\right\}.
$$ 
Based on the previous discussion, for $\phi\in\mathcal{G}(T)$, the position in the riskless asset at time $T+1$ is 
defined by
\begin{eqnarray}\label{eq:selffin}
X_T(\phi) &:=& z^0-\sum_{u=0}^T \phi_u S_u-\sum_{u=0}^T \lambda |\phi_u|^{\alpha}.
\end{eqnarray}
For simplicity, we also assume $z^0=0$ from
now on, i.e.\ portfolios start from nothing. 

To investigate 
the potential of realizing monetary profits, we focus on a risk-neutral objective: 
a linear utility function. Let $x_-:=\max\{-x,0\}$ for $x\in\mathbb{R}$. Define, for $T\in \mathbb{N}$,  
$$
\mathcal{A}(T):=\left\{\phi\in\mathcal{G}(T):\,E[(X_T(\phi))_-]<\infty\right\},
$$
the class of strategies starting from a zero initial position in both assets and ending at time $T+1$ in 
a cash only position with expected value greater than $-\infty$. The value of the problem we will consider is thus
$$
u(T):=\sup_{\phi\in\mathcal{A}(T)}E[X_T(\phi)].
$$
The investors's objective is to find $\phi$ which, at least
asymptotically as $T\to\infty$, achieves the same growth rate as $u(T)$.

\section{Asymptotically optimal investment}\label{aso}




First we introduce assumptions on the price process and its dependence structure.

\begin{assumption}\label{ass} Let $Z_t$, ${t \in \mathbb{Z}}$ 
be a real-valued, zero-mean stationary Gaussian process which will represent price increments.
Let $\mathcal{F}_{n}:=\sigma(Z_{i},\ i\leq n)$ for $n\in\mathbb{N}$.
Let $r(t):=\mathrm{cov}(Z_{0},Z_{t})$, $t\in\mathbb{Z}$ denote its covariance function.
We assume that there exists $T_0>0$ and $J_1, J_2<0$ such that  
for all $t\geq T_0$,
\begin{equation}\label{man}
J_1 t^{\chi} \leq r(t) \leq J_2 t^{\chi}
\end{equation}
is satisfied
for some parameter $\chi \in \left(-2,-1\right)$. Furthermore, 
\begin{equation}\label{sumnul}
\sum_{t \in \mathbb{Z}} r(t)=0.	
\end{equation} 
Let us introduce the adapted price process defined by $S_0=0$ 
and $S_t=S_{t-1}+Z_t$, $t\geq 1$.
\end{assumption} 

\begin{remark}\label{woman} {\rm Properties \eqref{man} and \eqref{sumnul} express that $Z$ is a \emph{process with negative memory},
see Definition 1.1.1 on page 1 of \cite{gks}.
When $Z_{t}$, $t\in\mathbb{Z}$ are the increments of a FBM with Hurst parameter $H<1/2$, then \eqref{man} is
satisfied with $\chi:=2H-2$.}
\end{remark}

The next theorem is our main result: it provides the explicit form
of an (asymptotically) optimal strategy and determines its expected asymptotic growth rate.

\begin{theorem}\label{theo} 
Let Assumption \ref{ass} be in force. If $\lambda$ is small enough then
maximal expected profits satisfy
\begin{equation}\label{M1}
\limsup_{T\to\infty} \frac{u(T)}{T^{\left(\frac{\chi}{2}+1\right)\left(1+\frac{1}{\alpha-1}\right)+1}}< \infty
\end{equation}
and the strategy
\begin{equation}\label{M05}
\phi_t(T,\alpha):=
\begin{cases}
-\mathrm{sgn}(S_t)|S_t|^{\frac{1}{\alpha-1}}, & 0\leq t \leq 3 \lfloor T/6 \rfloor , \\
-\frac{1}{3 \lfloor T/6 \rfloor}\sum_{s=0}^{3 \lfloor T/6 \rfloor}\phi_s, & 3 \lfloor T/6 \rfloor  < t\leq 6 \lfloor T/6 \rfloor, \\
\ \ \ \ \ \ \ \ \ \ \ \ \ 0 , & \mathrm{ otherwise}
\end{cases}
\end{equation}
satisfies
\begin{equation}\label{M2}
\liminf_{T\to\infty}\frac{EX_T(\phi(T,\alpha))}{T^{\left(\frac{\chi}{2}+1\right)\left(1+\frac{1}{\alpha-1}\right)+1}}>0.
\end{equation}
\end{theorem}

\begin{remark}
The strategy above builds on the following intuition.  In a market with friction one can not sell or buy with arbitrary speeds.  Such behavior is punished in superlinear price-impact models, strategies that are not trading assets gradually can generate losses that ruin an otherwise profitable investment.  Thus, liquidation must also be done at a careful pace. Our strategy operates as follows. On the first half of the given timeline it trades the underlying in a contrarian manner, that is going short when prices are high and entering long positions when low.  It is intuitively clear that, due to the superlinear nature of friction, liquidation is best done with a constant speed. This is reflected in our strategy on the second half of the timeline. 
\end{remark}

\section{Proofs}

\subsection{General bounds for variance and covariance}\label{sec::variousbounds}

First we make some useful preliminary observations. Using stationarity of the increments of the process $S$, we have 
\begin{align}\label{a-1}
\begin{split}
\mathrm{var}(S_t) & = \mathrm{cov}(S_t,S_t)  =\mathrm{cov}(\sum_{j=1}^t S_j-S_{j-1},\sum_{i=1}^t S_i-S_{i-1})
\\ & = t \cdot \mathrm{var}(S_1-S_0) + 2\sum_{i=2}^t \sum_{j=1}^{i-1} \mathrm{cov}(S_j-S_{j-1},S_i-S_{i-1})
\\ & =t \cdot \mathrm{var}(S_1-S_0) + 2\sum_{i=2}^t \sum_{j=1}^{i-1} \mathrm{cov}(S_1-S_{0},S_{i-j+1}-S_{i-j})
\\ & = t \cdot r(0) + 2\sum_{i=2}^t \sum_{j=1}^{i-1} r(i-j).
\end{split}
\end{align}
Furthermore, for $s>t$ we similarly have
\begin{align}\label{a-2}
\mathrm{cov}(S_s-S_t,S_t) = \sum_{i=t+1}^s \sum_{j=1}^t r(i-j).
\end{align}
Observe also that we can write
\begin{equation}\label{a-3}
r(0)=-2 \sum_{j=1}^{\infty} r(j).
\end{equation}
Turning to the variances, we first obtain a convenient expression for them. 
Note that for $i >1$ we have 
\begin{equation}\label{pami}
\sum_{j=1}^{i-1} r(i-j) = r(i-1)+\ldots+r(1)=r(1)+\ldots+r(i-1)=\sum_{j=1}^{i-1} r(j).
\end{equation} 
Using the observations \eqref{pami}, (\ref{a-1}) and (\ref{a-3}), we have 
\begin{align*}
\mathrm{var}(S_t) & =-2t \sum_{j=1}^{t-1}r(j) -2t\sum_{j=t}^{\infty} r(j) + 2\sum_{i=2}^t\sum_{j=1}^{i-1} r(j),
  \end{align*}
 and algebraic manipulation of the summation operation  
 $\left( -2t \sum_{j=1}^{t-1} + 2\sum_{i=2}^t\sum_{j=1}^{i-1} \right)$ yields
 \begin{align*}
& -2t \sum_{j=1}^{t-1}  + 2\sum_{i=2}^t\sum_{j=1}^{i-1}
\\ & =-2t \left ( \sum_{j=1}^{T_0-1}+\sum_{j=T_0}^{t-1} \right )  + 2 \left ( \sum_{i=2}^{T_0-1}+\sum_{i=T_0}^t \right ) \sum_{j=1}^{i-1} 
\\ & = -2t \sum_{j=1}^{T_0-1}-2t\sum_{j=T_0}^{t-1} + 2 \sum_{i=2}^{T_0-1}\sum_{j=1}^{i-1}+2\sum_{i=T_0}^t \sum_{j=1}^{i-1} 
\\ & =-2t \sum_{j=1}^{T_0-1}-2t\sum_{j=T_0}^{t-1}   + 2\sum_{i=2}^{T_0-1}\sum_{j=1}^{i-1}+ 2 \sum_{j=1}^{T_0-1} + 2 \sum_{i=T_0+1}^t \left( \sum_{j=1}^{T_0-1} + \sum_{j=T_0}^{i-1} \right) 
\\ & =-2t \sum_{j=1}^{T_0-1}-2t\sum_{j=T_0}^{t-1}    + 2\sum_{i=2}^{T_0-1} \sum_{j=1}^{i-1} + 2 \sum_{j=1}^{T_0-1} + 2 \sum_{i=T_0+1}^t \sum_{j=1}^{T_0-1}  + 2 \sum_{i=T_0+1}^t \sum_{j=T_0}^{i-1} 
\\ & =-2t \sum_{j=1}^{T_0-1}    + 2\sum_{i=2}^{T_0-1} \sum_{j=1}^{i-1} + 2 \sum_{j=1}^{T_0-1} + 2 \sum_{i=T_0+1}^t \sum_{j=1}^{T_0-1}-2t\sum_{j=T_0}^{t-1}  + 2 \sum_{i=T_0+1}^t \sum_{j=T_0}^{i-1},
 \end{align*}
 where the last line is only a reordering of terms.
Setting  $C_1=\sum_{j=1}^{T_0-1}r(j)$, $C_2=\sum_{i=2}^{T_0-1} \sum_{j=1}^{i-1} r(j)$ and $C_3=2(C_2-(T_0-1)C_1)$, 
the above calculation gives
\begin{align}\label{a-4}
\begin{split}
\mathrm{var}(S_t) & = - 2tC_1 +  2C_2 + 2C_1 + 2 (t-T_0) C_1 
+\left (-2t\sum_{j=t}^{\infty}  -2t \sum_{j=T_0}^{t-1}  + 2 \sum_{i=T_0+1}^t  \sum_{j=T_0}^{i-1}  \right) r(j)
\\ & = C_3 +\left (-2t\sum_{j=t}^{\infty} -2t \sum_{j=T_0}^{t-1} + 2 \sum_{i=T_0+1}^t  \sum_{j=T_0}^{i-1}  \right) r(j)
\end{split}
\end{align}

From now on we will work with the parameter $H:=\frac{\chi}{2}+1$ for convenience. This parameter
choice also refers back to the case of
FBMs, see Remark \ref{woman}. 

Now we are ready to present three lemmas, providing a lower and an upper bound for the variance and 
an upper bound for the covariance.

\begin{lemma}\label{lem1}
There exist $T_1\in\mathbb{N}$ and $B_1>0$ such that for all $t \geq T_1$ we have 
\begin{align*}
\mathrm{var}(S_t) \geq B_1 t^{2H}.
\end{align*}
\end{lemma}
\begin{proof}
Using properties induced by the choice of $T_0$ in Assumption \ref{ass} first note that
\begin{align*}
& \left ( -2t \sum_{j=T_0}^{t-1}  + 2 \sum_{i=T_0+1}^t  \sum_{j=T_0}^{i-1}  \right) r(j)
\\ & \geq \left ( -2t \sum_{j=T_0}^{t-1} + 2 (t-T_0) \sum_{j=T_0}^{t-1}  \right) r(j)
\\ & = -2T_0 \sum_{j=T_0}^{t-1}  r(j) \geq 0.
\end{align*}
Also notice that
\begin{align*}
-2t\sum_{j=t}^{\infty} r(j) & \geq -2J_2t\sum_{j=t}^{\infty} j^{2H-2}  \geq  -2J_2t\int_t^{\infty}u^{2H-2}du
\\ & =  -2J_2t \frac{1}{2H-1} \left( - t^{2H-1} \right ) =   \frac{2J_2}{2H-1} t^{2H}.
\end{align*}
Using these and (\ref{a-4})
\begin{align*}
\mathrm{var}(S_t) \geq C_3 + \frac{2J_2}{2H-1} t^{2H}.
\end{align*}
The threshold $T_1$ and the constant $B_1$ can be explicitly calculated in terms of the constants present in the above expression. This completes the proof.
\end{proof}

\begin{lemma}\label{lem2}
There exist $T_2\in\mathbb{N}$ and $B_2>0$ such that for all $t \geq T_2$ we have 
\begin{align*}
\mathrm{var}(S_t) \leq B_2 t^{2H}.
\end{align*}
\end{lemma}
\begin{proof}
First note that algebraic manipulation of the operation 
$\left ( -2t \sum_{j=T_0}^{t-1}  + 2 \sum_{i=T_0+1}^t \sum_{j=T_0}^{i-1}  \right)$ yields
\begin{align*}
& -2t \sum_{j=T_0}^{t-1}  + 2 \sum_{i=T_0+1}^t  \sum_{j=T_0}^{i-1}  = -2(t-T_0+T_0) \sum_{j=T_0}^{t-1}  + 2 \sum_{i=T_0}^{t-1}  \sum_{j=T_0}^{i} 
\\ & = -2 \sum_{i=T_0}^{t-1} \sum_{j=T_0}^{t-1}  + 2 \sum_{i=T_0}^{t-1}  \sum_{j=T_0}^{i}  -2T_0\sum_{j=T_0}^{t-1} = -2 \sum_{i=T_0}^{t-1} \left( \sum_{j=T_0}^{t-1}  -  \sum_{j=T_0}^{i} \right)  -2T_0\sum_{j=T_0}^{t-1}
\\ & =-2 \sum_{i=T_0}^{t-1}  \sum_{j=i+1}^{t-1}   -2T_0\sum_{j=T_0}^{t-1}.
\end{align*} 
By Assumption  \ref{ass}, this implies  
\begin{align*}
&\left( -2t   \sum_{j=T_0}^{t-1} + 2 \sum_{i=T_0+1}^t  \sum_{j=T_0}^{i-1}  \right) r(j) \leq -2 J_1 \left( \sum_{i=T_0}^{t-1}  \sum_{j=i+1}^{t-1} j^{2H-2}   +T_0\sum_{j=T_0}^{t-1} j^{2H-2} \right)
\\ &   \leq -2 J_1 \left( \sum_{i=T_0}^{t-1}  \int_{i}^{t-1} u^{2H-2} du   +T_0\int_{T_0-1}^{t-1} u^{2H-2} du \right)  
\\ & = - \frac{2 J_1}{2H-1} \left( \sum_{i=T_0}^{t-1}  \left( (t-1)^{2H-1} - i^{2H-1} \right)  +T_0 \left( (t-1)^{2H-1} - (T_0  - 1)^{2H-1} \right) \right)
\\ & = - \frac{2 J_1}{2H-1} \left( t (t-1)^{2H-1} - \sum_{i=T_0}^{t-1}   i^{2H-1}  - T_0 (T_0  - 1)^{2H-1} \right)
\\ & \leq  \frac{2 J_1}{2H-1} \sum_{i=T_0}^{t-1}   i^{2H-1} +  \frac{2 J_1}{2H-1} T_0 (T_0  - 1)^{2H-1}
\\ & \leq  \frac{2 J_1}{2H(2H-1)} ( (t-1)^{2H} - (T_0-1)^{2H} ) +  \frac{2 J_1}{2H-1} T_0 (T_0  - 1)^{2H-1}
\\ & \leq  \frac{2 J_1}{2H(2H-1)} t^{2H} +  \frac{2 J_1}{2H-1} T_0 (T_0  - 1)^{2H-1}.
\end{align*}
To proceed observe that, using the asymptotics in Assumption \ref{ass}, for $t>2$ we have
\begin{align*}
-2t \sum_{j=t}^{\infty} & r(j)  \leq - 2 J_1  t \sum_{j=t}^{\infty} j^{2H-2} \leq - 2 J_1  t \int_{t-1}^{\infty} u^{2H-2} du  
\\ & = \frac{2 J_1 t}{2H-1} (t-1)^{2H-1} \leq \frac{2 J_1 t}{2H-1} (t-t/2)^{2H-1}  
\\ & =\frac{2^{2-2H} J_1 }{2H-1}t^{2H}.
\end{align*} 
These results yield  for $t>\max(2,T_0)$, using again (\ref{a-4}), that
\begin{align}\label{a1}
\begin{split}
\mathrm{var}(S_t) & \leq C_3 + \left( \frac{2 J_1}{2H(2H-1)} + \frac{2^{2-2H} J_1 }{2H-1} \right) t^{2H} +  \frac{2 J_1}{2H-1} T_0 (T_0  - 1)^{2H-1}
\end{split}
\end{align}
The threshold $T_2$ and the constant $B_2$ could again be explicitly given. The proof is complete.
\end{proof}
\noindent We proceed with the lemma controlling the covariance $ \mathrm{cov}(S_s-S_t,S_t)$.
\begin{lemma}\label{lem3} 
There exist $T_3\in\mathbb{N}$ and $D_1, D_2>0$ such that
$$
\mathrm{cov}(S_s-S_t,S_t)\leq D_1 \ \mbox{  for all  } \ s>t>T_{3}.
$$ 
For a fixed $v>1$, define 
$$U(v):=J_2 \left(2H\right)^{-1}(2H-1)^{-1}  \left(  1-\left (    v^{2H}  -  (v-1)^{2H} \right) \right).$$
Then
$$
\mathrm{cov}(S_s-S_t,S_t)\leq D_2 - U(v) t^{2H}<0\mbox{ holds for all }s>t>T_{3}\mbox{ satisfying }\frac{s}{t}>v.$$
There exists ${K}>1$ and $T_4\in\mathbb{N}$ such that
$$
\mathrm{cov}(S_s-S_t,S_t)\leq 0\mbox{ for all }s>t>T_{4}\mbox{ satisfying }s-t>{K}.
$$
\end{lemma}
\begin{proof}
\noindent Let us set 
$$
C_4=\sum_{j=-T_0+1}^0 \sum_{i=1}^{1+T_0}r(i-j),\quad{}
C_5=J_2\sum_{j=-T_0+1}^0 \sum_{i=1}^{1+T_0} (i-j)^{2H-2},
$$ 
and define $C_6=C_4-C_5$. Note that, for each $t\in\mathbb{N}$,
$C_4=\sum_{j=t-T_0+1}^t \sum_{i=t+1}^{t+1+T_0}r(i-j),$ and
$C_5=J_2\sum_{j=t-T_0+1}^t \sum_{i=t+1}^{t+1+T_0} (i-j)^{2H-2}$. For $t>T_0$, we have
\begin{align}\label{b0}
\begin{split}
\mathrm{cov}(S_s&-S_t,S_t)  =  \sum_{j=1}^t \sum_{i=t+1}^s r(i-j) 
\\ & \leq  C_6+ J_2  \sum_{j=1}^t \sum_{i=t+1}^s (i-j)^{2H-2} \leq C_6 +  J_2  \sum_{j=1}^t \int_{t+1-j}^{s+1-j} u^{2H-2} du
\\ & \leq C_6 +  \frac{J_2 }{2H-1}  \sum_{j=1}^t \left ( (s+1-j)^{2H-1} - (t+1-j)^{2H-1}  \right )
\\ & = C_6 +   \frac{J_2}{2H-1}  \sum_{j=1}^t  (s+1-j)^{2H-1} -  \frac{J_2 }{2H-1}  \sum_{j=1}^t   (t+1-j)^{2H-1}
\\ & \leq  C_6 + \frac{J_2 }{2H-1} \int_{s-t}^s u^{2H-1} du  -  \frac{J_2 }{2H-1}  \int_{1}^{t+1} u^{2H-1} du
\\ & = C_6 +   \frac{J_2 }{2H(2H-1)} \left ( s^{2H} - (s-t)^{2H} \right)   -  \frac{J_2 }{2H(2H-1)}   \left ( (t+1)^{2H} - 1 \right)
\\ & = C_6 +   \frac{J_2 }{2H(2H-1)} \left ( s^{2H} - (s-t)^{2H} -   \left ( (t+1)^{2H} - 1 \right) \right ).
\\ & =: C_6 +   C_7 \left ( s^{2H} - (s-t)^{2H} -   \left ( (t+1)^{2H} - 1 \right) \right ).
\end{split}
\end{align}
Since $s\geq t+1$ the expression $C_7 \left ( s^{2H} - (s-t)^{2H} -   \left ( (t+1)^{2H} - 1 \right) \right )$ is non-positive, which yields
\begin{align*}
\mathrm{cov}(S_s&-S_t,S_t)   \leq  C_6,
\end{align*}
proving the first statement of the lemma. Now, for all $v>1$ the property  $\frac{s}{t}>v$ - together with the previous constraint of $t>T_0$ - further implies
\begin{align}\label{b2.5}
\begin{split}
\mathrm{cov}(S_s&-S_t,S_t)  \leq  C_6 +   C_7 \left ( s^{2H} - (s-t)^{2H} -   \left ( (t+1)^{2H} - 1 \right) \right )
\\ & \leq C_6 +   C_7\left ( (v^{2H} - (v-1)^{2H} -1)t^{2H} + 1  \right )
\\ & = C_6 + C_7 +   C_7(v^{2H} - (v-1)^{2H} -1)t^{2H}.
\end{split}
\end{align}
Obviously, for large enough $t$ the bound becomes strictly negative, proving the second statement. Now, assuming  $s-t\geq K>1$ beside $t>T_0$ we have
\begin{align}\label{b2}
\begin{split}
\mathrm{cov}(S_s&-S_t,S_t)  \leq C_6 +   C_7\left ( (t+K)^{2H} - K^{2H} -   \left ( (t+1)^{2H} - 1 \right) \right )
\\ & = C_6 -    C_7\left (  K^{2H}  - 1 \right) + C_7 \left( ( t+K)^{2H} -   (t+1)^{2H} \right)
\\ & \leq C_6 -    C_7\left (  K^{2H}  - 1 \right) + C_7 2H K t^{2H-1}.
\end{split}
\end{align}
This shows that $K$ can be chosen so large that $C_6 -    C_7\left (  K^{2H}  - 1 \right)<0$ and then, since $2H-1<0$, a threshold $T_4$ - depending on $K$ - for the variable $t$ can be specified so that $$C_6 -    C_7\left (  K^{2H}  - 1 \right) + C_7 2H K t^{2H-1} \leq 0$$ whenever $t$ exceeds the threshold, proving the third statement, completing the proof of the lemma.
\end{proof}

\subsection{Key estimates}
Define  
$$
\rho(s,t):=\frac{\mathrm{cov}(S_s,S_t)}{\mathrm{var}(S_t)}=\frac{\mathrm{cov}(S_s-S_t,S_t)}{\mathrm{var}(S_t)} + 1,\ 
s\in\mathbb{N},\ t\in\mathbb{N}\setminus\{0\}.
$$

\begin{lemma}\label{LEM}
There exist $\bar{T}\in\mathbb{N}$ and constants $R >0$, $K>1$, $\eta \in (1/2,1)$ and  $\varepsilon>0$ such that 
\begin{enumerate}
\item $\rho(s,t)<1+R, \ \mbox{for all} \  t<s$;
\item $\rho(s,t)\leq 1, \ \mbox{whenever} \ \bar{T}<t<s \ \mbox{and}  \ s-t>K;$
\item For all $T\in\mathbb{N}$, $\rho(s,t)\leq 1-\varepsilon, \ \mbox{whenever}  \ \bar{T}<t<\frac{T}{2}<\eta T<s$.
Furthermore, one can also guarantee $T/2+K<\eta T$ in this case.
\end{enumerate}
\end{lemma}

\begin{proof}[Proof of Lemma \ref{LEM}]
Let $B_2$, $U(\cdot)$,  $T_{1}$, $T_{2}$, $T_3$, $T_4$, $D_1$, $D_2$ and  $K$ be as in Lemma \ref{lem2} and Lemma \ref{lem3}. Choose $T'>\max\{T_{1},T_{2},T_3\}$ so large that $\frac{ D_2   }{ B_2 }(T')^{-2H}  - \frac{ U(4/3)   }{ B_2 } < 0$
and set $\eta:=2/3$. 
Lemma \ref{lem2} and Lemma \ref{lem3} now show that whenever $T'<t<T/2$ and $s \in (\eta T, T)$, we have
\begin{align}\label{b3}
\begin{split}
& \frac{\mathrm{cov}(S_s-S_t,S_t)}{\mathrm{var}(S_t)} \leq 
\frac{ D_2   }{ B_2 }t^{-2H}  - \frac{ U(4/3)   }{ B_2 }
\leq 
\frac{ D_2   }{ B_2 }(T')^{-2H}  - \frac{ U(4/3)   }{ B_2 },
\end{split}
\end{align}
which yields $\rho(s,t)\leq 1-\varepsilon,$ where $\varepsilon = -\frac{ D_2   }{ B_2 }(T')^{-2H}  + 
\frac{ U(4/3)   }{ B_2 }$.
Lemma \ref{lem3} shows that $t>T_4 $, ensures that $s-t>K$ implies $\rho(s,t) \leq 1$.
Finally, set $\bar{T}=\max\{ T', T_{4}, 3K \}$. It is clear 
-- using (\ref{b0}) in the proof of Lemma \ref{lem3} -- 
that for fixed $t$, the function $(s,t) \mapsto \rho(s,t)$ is bounded. 
So let $D_1'=\max_{0<t<\bar{T}} \sup_{s\geq 0}\rho(s,t)$ and define $R = \max\{D_1,D_1'\}-1$
It remains to guarantee $T/2+K<\eta T$ but this follows since $\bar{T}<t<T/2$ implies $T>6K$. The quantities $\eta$, $\bar{T}$, $R$, $K$ and $\varepsilon$ constructed above 
fulfill all the requirements.
\end{proof} 




\begin{proof}[Proof of Theorem \ref{theo}.]
\noindent First we determine the maximal expected growth rate of portfolios. Let us define $$Q(T) = \sum_{t=0}^T E | S_t |^{\frac{\alpha}{\alpha-1}}.$$ Let $G(x):=\lambda|x|^{\alpha}$, $x\in\mathbb{R}$ and denote its Fenchel-Legendre conjugate
\begin{equation}
G^*(y) := \sup_{x\in\mathbb{R}} (xy-G(x))= 
\frac{\alpha-1}{\alpha} \alpha^{\frac 1{1-\alpha}}\lambda^{\frac{1}{1-\alpha }} |y|^{\frac{\alpha }{\alpha -1}},
\qquad
y\in\mathbb{R}.
\end{equation}
By definition of $G^{*}$, for all $\phi\in\mathcal{G}(T)$,
$$
X_T(\phi) \leq \sum_{t=0}^T G^*(-S_t)    =   C\sum_{t=0}^T |S_t|^{\alpha/(\alpha-1)}
$$
for some $C>0$ and hence
\begin{equation}
EX_T(\phi)\leq C Q(T)<\infty.
\end{equation}
Note that this bound is independent of $\phi$. Using Lemma \ref{lem2} it holds that
\begin{align}\label{e1}
\begin{split}
Q(T) & =C_{\frac{\alpha}{\alpha-1}} \sum_{t=0}^T \mathrm{var}(S_t)^{\frac{\alpha}{2(\alpha-1)}}
\\ &  \leq C_{\frac{\alpha}{\alpha-1}} \sum_{t=0}^{T_2-1} \mathrm{var}(S_t)^{\frac{\alpha}{2(\alpha-1)}} + C_{\frac{\alpha}{\alpha-1}} B_2 \sum_{t=T_2}^T t^{\frac{H \alpha}{(\alpha-1)}}
\\ & \leq C_{\frac{\alpha}{\alpha-1}, T_2} +  C_{\alpha,H,B_2} T^{H \left ( 1+\frac{1}{\alpha-1} \right )+1}.
\end{split}
\end{align}
Thus the maximal expected profit grows as $T^{H \left ( 1+\frac{1}{\alpha-1} \right )+1}$  with the power of the horizon, this proves (\ref{M1}). Now, untill further notice, let $T$ be a multiple of $6$. With the strategy defined in (\ref{M05}), the dynamics takes the form 
\begin{align*}
X_T(\phi) = &  \sum_{t=0}^{T/2} |S_t|^{\frac{\alpha}{\alpha-1}} 
\\ & - \sum_{t=0}^{T/2} \lambda |S_t|^{\frac{\alpha}{\alpha-1}}
\\  & -  \frac{1}{T/2} \sum_{s=T/2+1}^T S_s \sum_{t=0}^{T/2} \mathrm{sgn}(S_t)|S_t|^{\frac{1}{\alpha-1}}  
\\ & - \frac{1}{T/2}  \sum_{s=T/2+1}^T \lambda \left|\sum_{t=0}^{T/2}\mathrm{sgn}(S_t)|S_t|^{\frac{1}{\alpha-1}} \right|^{\alpha}.
\end{align*}
In the above expression let us denote the four terms by 
$I_1(T)$, $I_2(T)$, $I_3(T)$, $I_4(T)$, respectively, so that 
$$X_T(\phi) = I_1(T) - I_2(T) - I_3(T) - I_4(T).$$
The upper bound constructed in (\ref{e1}) for $Q(T)$ right away gives us an upper estimate for $EI_1(T)$ as $EI_1(T) = Q(T/2) $. Using Lemma \ref{lem1}, we likewise present a lower estimate as 
\begin{align}\label{c2}
\begin{split}
Q(T/2)=E[I_1(T)] & = C_{\frac{\alpha}{\alpha-1}} \sum_{t=0}^{T/2} \mathrm{var}(S_t)^{\frac{\alpha}{2(\alpha-1)}}
\\  & \geq C_{\frac{\alpha}{\alpha-1}} \sum_{t=0}^{T_1-1} \mathrm{var}(S_t)^{\frac{\alpha}{2(\alpha-1)}} +  C_{\frac{\alpha}{\alpha-1}} B_1 \sum_{t=T_1}^{T/2} t^{\frac{H \alpha}{\alpha-1}}
\\ & \geq  C_{\frac{\alpha}{\alpha-1}, H, B_1,T_1} + C_{\frac{\alpha}{\alpha-1}, H, B_1} T^{H(1+\frac{1}{\alpha-1})+1},
\end{split}
\end{align}
To treat the terms  $I_2(T)$ and $I_4(T)$, note that with $\alpha>1$ the function $x \mapsto |x|^{\alpha}$ is convex, thus applying Jensen's inequality
\begin{align}\label{e2}
|EI_4(T)| \leq  E|I_2(T)| & = \lambda E \left [  \sum_{t=0}^{T/2}   |S_t|^{\frac{\alpha}{\alpha-1}} \right  ] =  \lambda\sum_{t=0}^{T/2} E   |S_t|^{\frac{\alpha}{\alpha-1}} = \lambda E[I_1(T)] = \lambda Q(T/2).
\end{align}
Controlling term $I_3(T)$ is done via exploiting a specific property of Gaussian processes, namely that $S_s$ for  $s>t$ can be decomposed as $S_s = \rho(s,t)S_t + W_{s,t}$, where $W_{s,t}$ is independent of $S_t$ and zero mean. With this, observe that
\begin{align}\label{c1}
\begin{split}
EI_3(T)=&\frac{1}{T/2}\sum_{s=T/2+1}^T  \sum_{t=0}^{T/2} E[\rho(s,t)S_t \mathrm{sgn}(S_t)|S_t|^{\frac{1}{\alpha-1}}] \\ &= \frac{1}{T/2}\sum_{s=T/2+1}^T  \sum_{t=0}^{T/2} E[\rho(s,t)|S_t|^{\frac{\alpha}{\alpha-1}}].
\end{split}
\end{align}
Let the constants $\bar{T}$,  $R$, $K$, $\eta=2/3$  and $\varepsilon$ be as in Lemma \ref{LEM}, and decompose the double sum in (\ref{c1}) as
\begin{align*}
\sum_{s=T/2+1}^T \sum_{t=0}^{T/2} = \sum_{s=T/2+1}^{T}\sum_{t=0}^{\bar{T}-1} + \sum_{s=T/2+1}^{T/2+K}\sum_{t=\bar{T}}^{T/2} + \sum_{s=T/2+K+1}^{\eta T}\sum_{t=\bar{T}}^{T/2} + \sum_{s=\eta T +1}^{T}\sum_{t=\bar{T}}^{T/2}
\end{align*}
Note that  applying the upper bound developed in Lemma \ref{LEM} to the double sum in  (\ref{c1}), the summand no longer depends on the running variable of the outer sum. Denoting $C_{\bar{T}} := \sum_{t=0}^{\bar{T}-1} E|S_t|^{\frac{\alpha}{\alpha-1}} $, this implies that
\begin{align*}
E I_3(T) & \leq   \left (  \sum_{t=0}^{T/2} + R \sum_{t=0}^{\bar{T}-1} + \frac{2 R K}{T} \sum_{t=\bar{T}}^{T/2} -2\varepsilon  \left(1-\frac{2}{3}\right ) \sum_{t=\bar{T}}^{T/2} \right )  E|S_t|^{\frac{\alpha}{\alpha-1}} 
\\ & =   E[I_1(T)] + \left ( R \sum_{t=0}^{\bar{T}-1} + \frac{2 R K}{T} \sum_{t=\bar{T}}^{T/2} - \frac{2 \varepsilon}{3}   \sum_{t=\bar{T}}^{T/2} \right )  E|S_t|^{\frac{\alpha}{\alpha-1}} 
\\ & = E[I_1(T)] + \left ( \left (  R + \frac{2\varepsilon}{3} -\frac{2 R K}{T} \right ) \sum_{t=0}^{\bar{T}-1} + \frac{2 R K}{T} \sum_{t=0}^{T/2} -\frac{2\varepsilon}{3} \sum_{t=0}^{T/2} \right )  E|S_t|^{\frac{\alpha}{\alpha-1}} 
\\ & = \left ( 1 -\frac{2\varepsilon}{3}  \right ) E[I_1(T)]
  + \left( R+  \frac{2 \varepsilon}{3} - \frac{2RK}{T}   \right) C_{\bar{T}}  
 + \frac{2 R K}{T} E[I_1(T)] ,
\end{align*}
so we have
\begin{align*}
E[I_1(T)]-E[I_3(T)] & \geq   \frac{2\varepsilon}{3} E[I_1(T)] 
- \left( R+ \frac{2 \varepsilon}{3} - \frac{2RK}{T}   \right) C_{\bar{T}}
 - \frac{2 R K}{T} E[I_1(T)]
\end{align*}
The above, using (\ref{e2}), boils down to
\begin{align*}
X_T(\phi) \geq \frac{2\varepsilon}{3} Q(T/2)
  - \left( R+ \frac{2 \varepsilon}{3} - \frac{2RK}{T}   \right) C_{\bar{T}} 
 - \frac{2 R K}{T} Q(T/2) - 2 \lambda Q(T/2)
\end{align*} 
Using (\ref{e1}) and (\ref{c2}), with $\lambda <  \varepsilon/3$, dividing through with $T^{H(1+\frac{1}{\alpha-1})+1}$ proves the statement in (\ref{M2}) with the constraint that the limiting operation runs through multiples of $6$.  Now let $T$ be general.  The same calculations can be done as above, with minor changes in the formulas corresponding to the upper and lower limits in summations according to taking the appropriate floor values.  That is, in the last inequality $Q(3 \lfloor T/6 \rfloor ) $ appears - instead of  $Q(T / 2 ) $ - and it grows in the order of $( 6 \lfloor T/6 \rfloor)^{H(1+\frac{1}{\alpha-1})+1}$, and using that $6 \lfloor T/6 \rfloor / T  $ tends to $1$ when $T$ is large, the proof of Theorem \ref{theo}, noting $\chi=2H-2$, is complete.
\end{proof}




\begin{thebibliography}{0}

\bibitem{almgren-chriss}
R. Almgren and N. Chriss.
\newblock Optimal execution of portfolio transactions.
\newblock \emph{Journal of Risk}, 3:5--40, 2001.

\bibitem{baillie}
R.~T. Baillie.
\newblock Long memory processes and fractional integration in econometrics.
\newblock {\em J. Econometrics}, 73:5--59, 1996.

\bibitem{bertsimas-lo}
D. Bertsimas and A. Lo.
\newblock Optimal control of execution costs.
\newblock \emph{Journal of Financial Markets}, 1:1--50, 1998.

\bibitem{cr}
F. Comte and E. Renault.
\newblock Long memory in continuous-time stochastic volatility models.
\newblock {\em Math. Finance}, 8:291--323, 1998.


\bibitem{dolinsky-soner}
Y. Dolinsky and H. M. Soner.
\newblock Duality and Convergence for Binomial Markets with Friction.
\newblock \emph{Finance Stoch.}, 17: 447--475, 2013.

\bibitem{garleanu-pedersen}
N. Garleanu and L. Pedersen.
\newblock {Dynamic trading with predictable returns and transaction costs}.
\newblock \emph{Journal of Finance}, 68:2309--2340, 2013.

\bibitem{gjr}
J. Gatheral, T. Jaisson and M. Rosenbaum. \newblock Volatility is rough.
\newblock\emph{Quantitative Finance}, 18:933--949, 2018. 

\bibitem{gks} L. Giraitis, H. L. Koul and D. Surgailis.
\newblock\emph{Large Sample Inference for Long Memory Processes.}
Imperial College Press, 2012.

\bibitem{guasoni2006no} P. Guasoni.
\newblock No arbitrage under transaction costs, with fractional {B}rownian motion and beyond.
\newblock \emph{Mathematical Finance}, 16:569--582, 2006.

\bibitem{gnr} P. Guasoni, Zs. Nika and M. R\'asonyi.{}
\newblock Trading fractional Brownian motion.
\newblock \emph{SIAM J. Financial Mathematics}, 10:769--789, 2019.
	
\bibitem{gr} P. Guasoni and M. R\'asonyi.{}
\newblock Hedging, arbitrage and optimality under superlinear friction.
\newblock \emph{Annals of Applied Probability}, 25:2066--2095, 2015.
	
\bibitem{hl} O. Hernandez-Lerma and J.-B. Lasserre.
\newblock \emph{Discrete-Time Markov Control Processes.}
\newblock Springer, 1996.
	
\bibitem{kyle}
A. S. Kyle.
\newblock Continuous auctions and insider trading.
\newblock \emph{Econometrica}, 29:1315--1335, 1985.

\bibitem{tim} T. Leung and B. Ward.
\newblock Tracking VIX with VIX Futures: Portfolio Construction and Performance.
\newblock \emph{Preprint}, 2019.\newblock arXiv:1907.00293
		
\bibitem{mandelbrot} B. B. Mandelbrot.
\newblock When can price be arbitraged efficiently? A limit to the validity of the random walk and martingale models.
\newblock \emph{The Review of Economics and Statistics}, 53:225--236, 1971.
  	
\bibitem{rogers} L. C. G. Rogers. Arbitrage with fractional {B}rownian motion.
\newblock \emph{Mathematical Finance}, 7:95--105. 1997.

\bibitem{rs}
L. C. G. Rogers and S. Singh.
\newblock The cost of illiquidity and its effects on hedging.
\newblock \emph{Math. Finance}, 20:597--615, 2010.

\bibitem{walter}
W. Schachermayer and J. Teichmann.
\newblock How close are the Option Pricing Formulas of Bachelier and Black-Merton-Scholes?
\newblock \emph{Mathematical Finance}, 18:155--170, 2008.

\end{thebibliography}
\end{document}